\documentclass[12pt]{amsart}
\usepackage{amsmath,amssymb,amsbsy,amsfonts,amsthm,latexsym,
                        amsopn,amstext,amsxtra,euscript,amscd,mathrsfs,color,bm}
                        
\usepackage{float} 
\usepackage[english]{babel}
\usepackage{mathtools}
\usepackage{todonotes}
\usepackage{url}
\usepackage[colorlinks,linkcolor=blue,anchorcolor=blue,citecolor=blue,backref=page]{hyperref}

\usepackage[norefs,nocites]{refcheck}

\usepackage{color}

\usepackage[norefs,nocites]{refcheck}
 
\begin{document}

\newtheorem{theorem}{Theorem}
\newtheorem{lemma}[theorem]{Lemma}
\newtheorem{lem}[theorem]{Lemma}
\newtheorem{thm}[theorem]{Theorem}
\newtheorem{claim}[theorem]{Claim}
\newtheorem{cor}[theorem]{Corollary}
\newtheorem{prop}[theorem]{Proposition}
\newtheorem{definition}{Definition}
\newtheorem{question}[theorem]{Question}
\newtheorem{conj}[theorem]{Conjecture}
\newtheorem{remark}[theorem]{Remark}
\newcommand{\hh}{{{\mathrm h}}}

\numberwithin{equation}{section}
\numberwithin{theorem}{section}

\def\sssum{\mathop{\sum\!\sum\!\sum}}
\def\ssum{\mathop{\sum\ldots \sum}}

\def \balpha{\boldsymbol\alpha}
\def \bbeta{\boldsymbol\beta}
\def \bgamma{{\boldsymbol\gamma}}
\def \bomega{\boldsymbol\omega}

\newcommand{\Res}{\mathrm{Res}\,}
\newcommand{\Gal}{\mathrm{Gal}\,}

\def\sssum{\mathop{\sum\!\sum\!\sum}}
\def\ssum{\mathop{\sum\ldots \sum}}
\def\dsum{\mathop{\sum\  \sum}}
\def\iint{\mathop{\int\ldots \int}}

\def\squareforqed{\hbox{\rlap{$\sqcap$}$\sqcup$}}
\def\qed{\ifmmode\squareforqed\else{\unskip\nobreak\hfil
\penalty50\hskip1em\null\nobreak\hfil\squareforqed
\parfillskip=0pt\finalhyphendemerits=0\endgraf}\fi}%%

%  use the AMS-Euler Fraktur fonts
%%%%%%%%%%%%%%%%%%%%%%%%%%%%%%%%%%
\newfont{\teneufm}{eufm10}
\newfont{\seveneufm}{eufm7}
\newfont{\fiveeufm}{eufm5}
%%%%%%%%%%%%%%%%%%%%%%%%%%%%%%%%%
%
%  allow automatic size selection in math mode
%
%%%%%%%%%%%%%%%%%%%%%%%%%%%%%%%%%
\newfam\eufmfam
     \textfont\eufmfam=\teneufm
\scriptfont\eufmfam=\seveneufm
     \scriptscriptfont\eufmfam=\fiveeufm
%%%%%%%%%%%%%%%%%%%%%%%%%%%%%%%%%
%
%  \frak works on a single symbol at a time...
%
\def\frak#1{{\fam\eufmfam\relax#1}}

\def\fK{\mathfrak K}
\def\fT{\mathfrak{T}}

\def\fA{{\mathfrak A}}
\def\fB{{\mathfrak B}}
\def\fC{{\mathfrak C}}
\def\fD{{\mathfrak D}}
\def\fM{{\mathfrak M}}

\newcommand{\sX}{\ensuremath{\mathscr{X}}}

\def\eqref#1{(\ref{#1})}

\def\vec#1{\mathbf{#1}}
\def\dist{\mathrm{dist}}
\def\vol#1{\mathrm{vol}\,{#1}}

\def\squareforqed{\hbox{\rlap{$\sqcap$}$\sqcup$}}
\def\qed{\ifmmode\squareforqed\else{\unskip\nobreak\hfil
\penalty50\hskip1em\null\nobreak\hfil\squareforqed
\parfillskip=0pt\finalhyphendemerits=0\endgraf}\fi}

\def\sA{\mathscr A}
\def\sB{\mathscr B}
\def\sC{\mathscr C}
\def\sD{\Delta}
\def\sE{\mathscr E}
\def\sF{\mathscr F}
\def\sG{\mathscr G}
\def\sH{\mathscr H}
\def\sI{\mathscr I}
\def\sJ{\mathscr J}
\def\sK{\mathscr K}
\def\sL{\mathscr L}
\def\sM{\mathscr M}
\def\sN{\mathscr N}
\def\sO{\mathscr O}
\def\sP{\mathscr P}
\def\sQ{\mathscr Q}
\def\sR{\mathscr R}
\def\sS{\mathscr S}
\def\sU{\mathscr U}
\def\sT{\mathscr T}
\def\sV{\mathscr V}
\def\sW{\mathscr W}
\def\sX{\mathscr X}
\def\sY{\mathscr Y}
\def\sZ{\mathscr Z}

%%%%%%%%%%%%%%%%%%%%%%%%%
% Alphabet calligraphie %
%%%%%%%%%%%%%%%%%%%%%%%%%
\def\cA{{\mathcal A}}
\def\cB{{\mathcal B}}
\def\cC{{\mathcal C}}
\def\cD{{\mathcal D}}
\def\cE{{\mathcal E}}
\def\cF{{\mathcal F}}
\def\cG{{\mathcal G}}
\def\cH{{\mathcal H}}
\def\cI{{\mathcal I}}
\def\cJ{{\mathcal J}}
\def\cK{{\mathcal K}}
\def\cL{{\mathcal L}}
\def\cM{{\mathcal M}}
\def\cN{{\mathcal N}}
\def\cO{{\mathcal O}}
\def\cP{{\mathcal P}}
\def\cQ{{\mathcal Q}}
\def\cR{{\mathcal R}}
\def\cS{{\mathcal S}}
\def\cT{{\mathcal T}}
\def\cU{{\mathcal U}}
\def\cV{{\mathcal V}}
\def\cW{{\mathcal W}}
\def\cX{{\mathcal X}}
\def\cY{{\mathcal Y}}
\def\cZ{{\mathcal Z}}
\newcommand{\rmod}[1]{\: \mbox{mod} \: #1}

\def\rE{\mathrm E}
\def\rS{\mathrm S}

\def\vr{\mathbf r}

\def\e{{\mathbf{\,e}}}
\def\ep{{\mathbf{\,e}}_p}
\def\em{{\mathbf{\,e}}_m}
\def\en{{\mathbf{\,e}}_n}

\def\Tr{{\mathrm{Tr}}}
\def\Nm{{\mathrm{Nm}}}
\def\supp{{\mathrm{supp}}}

\def\gmax{g}
\def\ve{\varepsilon}
\def\rot{\operatorname{rot}}
\def\ord{\operatorname{ord}}

\def\lcm{{\mathrm{lcm}}}
\def\lcm{{\mathrm{lcm}}}

\def \ovFp{\overline{\F}_p}

\def\({\left(}
\def\){\right)}
\def\fl#1{\left\lfloor#1\right\rfloor}
\def\rf#1{\left\lceil#1\right\rceil}

\definecolor{olive}{rgb}{0.3, 0.4, .1}
\definecolor{dgreen}{rgb}{0.,0.5,0.}

\def\mand{\qquad \mbox{and} \qquad}

%%%%%%%%%%%%%%%%%%%%%%%%%%%%%%%%%%%%%%%%%%%%%%%%%%%%%%%%
%%%%%%%%%%%%%%%%%%%%%%%%%%%%%%%%%%%%%%%%%%%%%%%%%%%%%%%%
%%%%%%%%%%%%%%%%%%%%%%%%%%%%%%%%%%%%%%%%%%%%%%%%%%%%%%%%
%%%%%%%%%%%%%%%%%%%%%%%%%%%%%%%%%%%%%%%%%%%%%%%%%%%%%%%%

%%%%%%%  END OF STANDARD STUFF %%%%%%%%%

%%%%%%%%%%%%%%%%%%%%%%%%%%%%%%%%%%%%%%%%%%%%%%%%%%%%%%%%
%%%%%%%%%%%%%%%%%%%%%%%%%%%%%%%%%%%%%%%%%%%%%%%%%%%%%%%%
%%%%%%%%%%%%%%%%%%%%%%%%%%%%%%%%%%%%%%%%%%%%%%%%%%%%%%%%
%%%%%%%%%%%%%%%%%%%%%%%%%%%%%%%%%%%%%%%%%%%%%%%%%%%%%%%
%%%%%%%%%%%
%%% Spell

\hyphenation{re-pub-lished}
\hyphenation{ne-ce-ssa-ry}

\parskip 4pt plus 2pt minus 2pt

\def\bfdefault{b}
\overfullrule=5pt

\def \F{{\mathbb F}}
\def \K{{\mathbb K}}
\def \L{{\mathbb L}}
\def \N{{\mathbb N}}
\def \Z{{\mathbb Z}}
\def \Q{{\mathbb Q}}
\def \R{{\mathbb R}}
\def \C{{\mathbb C}}
\def\Fp{\F_p}
\def \fp{\Fp^*}

\title[Polynomial Equations  in Subgroups]{
Polynomial Equations in Subgroups and Applications}

\author[S. Konyagin]{Sergei V.~Konyagin}
\address{Steklov Mathematical Institute,
8, Gubkin Street, Moscow, 119991, Russia} 
\email{konyagin@mi.ras.ru}

\author[I.   Shparlinski]{Igor E. Shparlinski}
\address{Department of Pure Mathematics, University of New South Wales\\
2052 NSW, Australia}
\email{igor.shparlinski@unsw.edu.au}

\author[I.   Vyugin]{Ilya V. Vyugin}
\address{Institute for Information Transmission Problems RAS\\
19, Bolshoy Karetny per., Moscow, 127051, Russia\\
and\\
Department of Mathematics, National Research University Higher School of Economics\\
6, Usacheva Street, Moscow, 119048, Russia\\
and\\
Steklov Mathematical Institute,\\
8, Gubkin Street, Moscow, 119991, Russia} 
\email{vyugin@gmail.com}

\begin{abstract} We obtain a new bound for the number 
of solutions to polynomial equations in cosets of multiplicative subgroups in finite fields, which generalises previous results of P.~Corvaja and U.~Zannier (2013). 
We also obtain a conditional improvement of  recent results of J.~Bourgain,  A.~Gamburd and P.~Sarnak (2016) and
S.~V.~Konyagin, S.~V.~Makarychev,  I.~E.~Shparlinski and  I.~V.~Vyugin (2019) on the structure 
%SK of solution to the reduction of the Markoff equation $x^2 +  y^2 + z^2 = 3 x yz$ 
of solutions to the reduction of the Markoff equation $x^2 +  y^2 + z^2 = 3 x yz$ 
modulo a prime $p$. 
\end{abstract}

\keywords{Polynomial equation, Markoff equation,  reduction modulo $p$}
\subjclass[2010]{11D79, 11T06}

\maketitle

\section{Introduction}

\subsection{Background and motivation}
 Bourgain, Gamburd and Sarnak~\cite{BGS1,BGS2} have recently initiated the study 
 of reductions modulo $p$ of  the set $\cM$ of {\it Markoff triples\/} $(x, y, z) \in \N^3$  which are positive integer solutions to the Diophantine equation 
\begin{equation}
\label{eq:Markoff}
x^2 +  y^2 + z^2 = 3 x yz, \qquad (x, y, z)\in \Z^3.
\end{equation}
Simple computation shows  that the map 
$$\cR_1:~(x, y, z) \mapsto (3 yz - x,  y, z)
$$ 
and similarly defined maps $\cR_2$, $\cR_3$ (which are all involutions),  send one Markoff triple to another.
Due to the symmetry of~\eqref{eq:Markoff}, the set $\cM$ is also invariant under  permutations  $ \Pi \in \rS_3$ 
of the components of $(x, y, z)\in \cM$.  It is also easy to check that the transformations 
$\cR_i$, $i=1,2,3$ and permutations $\Pi$ generate a group of transformations acting on $\cM$. 
 
 A celebrated  result of Markoff~\cite{Mar1,Mar2}  asserts that  all integer solutions to~\eqref{eq:Markoff} can 
 be generated  from the solution $(1,1,1)$ by using sequences of  the above transformations.

This naturally leads to the notion of the {\it functional graph\/} on Markoff triples with the ``root''  $(1,1,1)$,  
 and  edges
 $(x_1, y_1,z_1)  \to  (x_2, y_2,z_2)$, povided that  $(x_2, y_2,z_2) = \cT(x_1, y_1,z_1)$,  where 
\begin{equation}
\label{eq:Transf}
 \cT=\{\cR_1,\cR_2,\cR_3\}\cup \rS_3.
 \end{equation}
 In this terminology the  result of Markoff~\cite{Mar1,Mar2} asserts 
 then this graph is {\it connected\/}.

 Baragar~\cite[Section~V.3]{Bar} and, more recently, 
 Bourgain, Gamburd and Sarnak~\cite{BGS1,BGS2} conjecture that this property is preserved modulo all 
sufficiently large primes and the set of non-zero solutions $\cM_p$
to~\eqref{eq:Markoff} considered modulo $p$. In particular, this means that $\cM_p$ 
can be obtained from the set of Markoff triples  
$\cM$ reduced modulo $p$. 

 This conjecture, which we can also as write $\cM_p = \cM \pmod p$, means that the functional graph $\cX_p$ associated 
with the transformation~\eqref{eq:Transf} remains connected.  

Accordingly, if we define by $\cC_p\subseteq \cM_p$ the set of the triples in the largest connected component
of the above graph $\cX_p$,  then we can state:

\begin{conj}[Baragar~\cite{Bar}; Bourgain, Gamburd and Sarnak~\cite{BGS1,BGS2}] 
\label{conj:BGS} 
For every prime $p$  we have $\cC_p = \cM_p$.
\end{conj}

Bourgain, Gamburd and Sarnak~\cite{BGS1,BGS2} have obtained several major results towards 
Conjecture~\ref{conj:BGS}, see also~\cite{CGMP,dCIL,dCIM,GMR}. 
For example, by~\cite[Theorem~1]{BGS1} we have 
\begin{equation}
\label{eq:except set}
\#\(\cM_p \setminus \cC_p\) = p ^{o(1)}, \qquad \text{as} \ p \to \infty, 
\end{equation}
and also by~\cite[Theorem~2]{BGS1} we know that Conjecture~\ref{conj:BGS} holds for 
all but maybe at most $X^{o(1)}$ primes $p \le X$ as $X\to \infty$.

The bound~\eqref{eq:except set} has been improved in~\cite[Theorem~1.2]{KMSV} as 
\begin{equation}
\label{eq:except set improve}
\#\(\cM_p \setminus \cC_p\) \le \exp\((\log p)^{2/3+o(1)}\) , \qquad \text{as} \ p \to \infty
\end{equation}

Furthermore, Bourgain, Gamburd and Sarnak~\cite{BGS1,BGS2} have also proved that the size of any connected component
of the graphs $\cX_p$ is at least
\begin{equation}
\label{eq:large comp}
\#\cX_p\ge c (\log p)^{1/3},
\end{equation}
for some absolute constant $c > 0$. 
In turn,  the bound~\eqref{eq:large comp}has been improved in~\cite[Theorem~1.3]{KMSV}  as 
\begin{equation}
\label{eq:large comp improve}
\#\cX_p\ge c (\log p)^{7/9}. 
\end{equation} 

The improvements in~\eqref{eq:except set improve} and~\eqref{eq:large comp improve} are based 
on a bound of Corvaja and Zannier~\cite[Corollary~2]{CoZa},  
 on the number of solutions to the equation 
$$
P(u,v) = 0, \quad  (u,v)\in \cG_1\times \cG_2
$$
where $P$ is a bivariate absolutely irreducible polynomial over the finite field $\F_p$ of $p$ elements 
and $\cG_1,\cG_2  \subseteq \ovFp$  are multiplicative groups in the algebraic closure of $\F_p$,
see also~\cite{GaVo, HBK, MaVy, ShkVyu} for some related results. 

Motivated by the above results and connections, here we
\begin{itemize}
\item  derive  a new bound on the number of solutions in subgroups to a systems of several polynomials which 
covers  under a unified setting the results of~\cite{CoZa, MaVy, ShkVyu}; 
\item obtain an improvement of~\eqref{eq:except set improve} under a very plausible conjecture 
on the number of solutions in subgroups of  some particular equation over $\F_p^*$. 
\end{itemize}

\subsection{New results}

As before, for a prime $p$ we use $\ovFp$ to denote 
the algebraic closure of the finite field $\F_p$ of $p$ elements. 

For a bivariate irreducible polynomial
\begin{equation}
\label{eq:P}
P(X,Y)=\sum_{i+j\le  d} a_{ij}X^{i}Y^{j} \in \ovFp[X,Y]
\end{equation}
of total degree $\deg P \le d$, we 
define $P^{\sharp}(X,Y)$ as the homogeneous polynomial  of degree 
$d^{\sharp}=\min\{ i+j:~a_{ij}\ne 0\}$ 
given by 
\begin{equation}
\label{eq:Pmin}
P^{\sharp}(X,Y)=\sum_{i+j=d^{\sharp}} a_{ij}X^{i}Y^{j}.
\end{equation}
We  also consider the set of polynomials $\cP$:
$$
\cP= 
\{ P(\lambda X,\mu Y) \mid \lambda,\mu
\in\overline{\mathbb{F}}_p^* \}.
$$
Define $g$ as the greatest common divisor of the following set 
of differences
\begin{equation}
\label{eq:gk}
g=\gcd\{ i_1+j_1-i_2 -j_2 ~:~ a_{i_1,j_1}  a_{i_2,j_2}  \ne 0\}.  
\end{equation}

Given a   multiplicative  subgroup  $\cG \subseteq \ovFp$, we say that two polynomials 
$P, Q \in \ovFp[X,Y]$  are {\it $\cG$-independent\/} if 
there is no $(u,v)\in \cG^2$ and $\gamma\in \ovFp^*$ 
such that polynomials $P(X,Y)$ and $\gamma
Q(uX,vY)$ coincide. 

We now fix $h$ polynomials 
\begin{equation}
\label{eq:Pk}
P_k(X,Y) = P(\lambda_k X,\mu_k Y) \in\cP,\qquad  k=1,\ldots,h, 
\end{equation}
which are $\cG$-independent. 

The following result 
generalises a series of previous estimates of a similar type, 
see~\cite{CoZa, GaVo, HBK, MaVy, ShkVyu} and references therein. 

\begin{theorem}\label{thm:MV}
Suppose that $P$ is irreducible,
$$
\deg_X P = m \mand \deg_Y P = n
$$
and also  that $P^{\sharp}(X,Y)$ consists of at least two monomials. 
There exists a constant $c_0(m,n)$, depending only on $m$ and $n$, 
such that for any  multiplicative  subgroup  $\cG \subseteq \ovFp$ 
of order $t = \# \cG$  satisfying 
$$
\frac{1}{2}p^{3/4}h^{-1/4}\ge  t\ge \max\{h^2, c_0(m,n)\},
$$ 
 and 
 $\cG$-independent polynomials~\eqref{eq:Pk}  we have
 $$
\sum_{i=1}^h\#\left\{ (u,v)\in \cG^2 ~:~P_i(u,v)=0\right\} <12mn(m+n) \gmax h^{2/3}t^{2/3}.
$$
\end{theorem}

Our next result is conditional on the following:

\begin{conj}
\label{conj:eq_groups}
There exist constants $\ve_0>0$ and $A$ such that for any prime $p$,
any subgroup $\cG \subseteq \ovFp$ with $\#\cG\le p^{\ve_0}$,
and any elements $\alpha_{1,1},\alpha_{1,2},\alpha_{2,1},\alpha_{2,2} \in \ovFp$ satisfying 
\begin{equation}
\label{eq:conj-cond}
\alpha_{1,1}\neq0, \quad \alpha_{1,2}\neq0, \quad  \alpha_{1,1}\alpha_{2,2}-\alpha_{1,2}\alpha_{2,1}\neq0,
\end{equation}
the equation
\begin{equation}
\label{eq:conj}
\frac{\alpha_{1,1}u-\alpha_{1,2}}{\alpha_{2,1}u-\alpha_{2,2}} = v
\end{equation}
has at most $A$ solutions in $u,v\in\cG$.
\end{conj}  

\begin{remark} It is likely that the constant $A$ in Conjecture~\ref{conj:eq_groups}
cannot be taken less than $9$, even for $\cG \subseteq \Fp$ rather than for
$\cG \subseteq \ovFp$, see some heuristic arguments in Section~\ref{sec:comm}.  It is possible that this is optimal and 
Conjecture~\ref{conj:eq_groups} holds with $A=9$.
Also we have $\varepsilon_0\le 1/2$, see Section~\ref{sec:comm}.
\end{remark}

\begin{remark} It is easy to see that using the bound~\eqref{eq:except set improve} instead of~\eqref{eq:except set}
in the argument of the proof of Theorem~\ref{thm:except set2} immediately allows us to relax the condition of 
Conjecture~\ref{conj:eq_groups} to counting solutions in subgroups $\cG \subseteq \F_{p^2}$ 
of order $\#\cG\le  \exp\((\log p)^{2/3+\varepsilon_0}\)$. 
However we believe Conjecture~\ref{conj:eq_groups}  holds as stated. 
\end{remark}  

\begin{theorem}
    \label{thm:except set2}
    If Conjecture~\ref{conj:eq_groups} holds for some $\ve_0$ and $A$, then
for sufficiently large $p$ we have
    $$\#\(\cM_p \setminus \cC_p\) \le (\log p)^B,$$
where $B = 16 \log A  + c$  for an absolute constant $c$. 
\end{theorem}

\section{Solutions to  polynomial equations in subgroups 
of finite fields}

%\subsection{Set-up}

\subsection{Stepanov's method}

Consider a polynomial $\Phi\in \overline{\mathbb{F}}_p[X,Y,Z]$ such that
$$
\deg_X\Phi<A,\quad \deg_Y\Phi<B,\quad \deg_Z\Phi<C,
$$
that is,
$$
\Phi(X,Y,Z)=\sum_{0 \le a <A} \sum_{0 \le b< B} \sum_{0 \le c<C}\omega_{a,b,c}X^aY^bZ^c. 
$$
We assume 
$$
A<t
$$
where $t= \# \cG$ is the order of the subgroup $\cG\subseteq \F_p^*$, 
and  consider the polynomial 
$$
\Psi(X,Y)=Y^t\Phi(X/Y,X^{t},Y^t). 
$$
Clearly
$$
\deg \Psi \le t + t(B-1)  +  t(C-1) = (B+C-1)t .
$$
We now fix some  $\cG$-independent polynomials~\eqref{eq:Pk} and define 
the sets
\begin{equation}
\label{eq: Fi E}
\cF_i =  \left(\lambda_i^{-1}\cG\times\mu_i^{-1} \cG\right), \quad i =1, \ldots, h, 
\mand \cE =  \bigcup_{i=1}^h \cF_i.
\end{equation}
We also consider the locus of singularity
\begin{align*}
\cM_{sing} = \bigl\{(X,Y) \mid X Y=P(X,Y)=0& \  \text{or} \\
    \frac{\partial{}}{\partial{Y}}P(X, Y)&= P(X,Y)=0\bigr\}.
\end{align*}

\begin{lemma}\label{lem:Msing}
Let $P(X,Y)$ be an irreducible polynomial of bi-degree 
$$
\left(\deg_X P, \deg_Y P\right) = (m,n)
$$ 
and let  $n \ge 1$. 
Then for the cardinality of the set $\cM_{sing}$ the following holds:
$$
\# \cM_{sing} \le (m+n)^2.
$$
\end{lemma}

\begin{proof}
If the polynomial $P(X,Y)$ is irreducible, then the polynomials $P(X,Y)$ and $\frac{\partial P}{\partial Y}(X,Y)$ are relatively prime. 
Thus the B\'ezout theorem yields the bound $L\le (m+n)(m+n-1)$, where $L$ is the number of roots of the system
$$
 \frac{\partial{}}{\partial{Y}}P(X, Y)= P(X,Y)=0.
$$
Actually, the number of $X$ with $P(X,0)=0$  
is less than or equal to $\deg_X P(X,Y)=m$,  the number of pairs $(0,Y)$ on the curve 
\begin{equation}
\label{eq:curve P}
P(X,Y)=0
\end{equation}
where $P$ is given by~\eqref{eq:P},  is less than or equal to $\deg_Y P(X,Y)=n$. 
The total numbers  of such pairs is at most $L+ m+n \le (m+n)^2$.
\end{proof}

Assume  that the polynomial $\Psi$ and the $\cG$-independent polynomials~\eqref{eq:Pk}
satisfy the following conditions:
\begin{itemize}
\item all pairs  in the set 
$$
\left\{ (X,Y) \in \cE\setminus \cM_{sing} \mid  P(X,Y)=0 \right \}
$$ 
are zeros of orders at least $D$ of the
function $\Psi(X,Y)$ on the curve~\eqref{eq:curve P}; 

\item the polynomials $\Psi(X,Y)$ and $P(X,Y)$ are relatively prime. 
\end{itemize}

If these conditions are satisfied then the  {\it B\'ezout  theorem\/} gives us the upper bound 
$D^{-1} \deg \Psi \deg P+\#\cM_{sing}$ for the number of roots $(x,y)$ of the system
$$
\Psi(X,Y) =P(X,Y)=0, \qquad (X,Y) \in \cG.
$$
Since the polynomials $P_k$ are  $\cG$-independent,  
the sets $\cF_k$ are disjoint and also there is a one-to-one correspondence
between the zeros: 
\begin{align*}
P_k(X,Y)=0, &\ (X,Y)\in\cG^2, \\
& \Longleftrightarrow P(u,v)=0,   \ (u,v) = 
(\lambda_k^{-1}X, \mu_k^{-1} Y)\in\cF_k.
\end{align*}
Therefore,  we obtain the bound 
\begin{equation}
\begin{split}
\label{eq:Nh}
N_h& \le \frac{\deg \Psi  \cdot  \deg P}{D}+\# \cM_{sing}\\
& \le  \frac{(m+n) (B+C-1)t}{D}+\# \cM_{sing} 
\end{split}
\end{equation}
on the total number  of zeros of $P_k$ in $\cG^2$, $k=1,\ldots,h$: 
%\begin{equation}
%\label{eq:Nh}
$$
N_h=\sum_{k=1}^h\#\{ (u,v)\in \cG^2 ~:~P_k(u,v)=0\} .
$$
%\end{equation}

For completeness,  we present proofs of several results from~\cite{MaVy}
which we use here as well.

\subsection{Some divisibilities and non-divisibilities}

We begin with some simple  preparatory results on the divisibility of 
polynomials.

\begin{lemma}\label{lem-nonzero}
Suppose that $Q(X,Y)\in \F_p[X,Y]$ is an irreducible polynomial such that 
$$
Q(X,Y)\mid\Psi(X,Y)
$$ 
and $Q^{\sharp}(X,Y)$ consists of at least two monomials. 
Then
$$
Q^{\sharp}(X,Y)^{\lfloor t/e \rfloor }\mid\Psi^{\sharp}(X,Y),
$$
where $Q^{\sharp}(X,Y)$ and $\Psi^{\sharp}(X,Y)$ are defined as in~\eqref{eq:Pmin} and $e$ is defined as $g$ in~\eqref{eq:gk}, with respect to $Q(X,Y)$ instead of $P(x,y)$. 
\end{lemma}

\begin{proof}  Consider $\rho \in \cG$ and substitute $X=\rho \widetilde{X}$ and $Y=\rho \widetilde{Y}$ in the polynomials $Q(X,Y)$ and $\Psi(X,Y)$. Then 
$$
Q(X,Y)\longmapsto Q_\rho (\widetilde{X},\widetilde{Y})=Q(\rho \widetilde{X},\rho \widetilde{Y}),
$$ 
and
\begin{align*}
\Psi(X,Y) &= \Psi(\rho \widetilde{X},\rho \widetilde{Y})\\
& =(\rho \widetilde{Y})^t\Phi((\rho \widetilde{X})/(\rho \widetilde{Y}),
(\rho \widetilde{X})^{t},(\rho \widetilde{Y})^{t})
=\Psi(\widetilde{X},\widetilde{Y}),
\end{align*}
 because $\rho ^t=1$. Hence for any $\rho \in \cG$ we have
$$
Q_\rho (X,Y)\mid \Psi(X,Y), 
$$ 
and we also note that $Q_\rho (X,Y)$ is irreducible.

Since $Q^{\sharp}(X,Y)$ contains at least two monomials $e\geqslant 1$ is a correctly defined and there exist at least $s = \fl{t/e}$ elements $\rho _1,\ldots,\rho _{s}\in \cG$ such that 
\begin{equation}\label{ratio}
Q_{\rho _i}(X,Y)/Q_{\rho _j}(X,Y)\notin \overline{\mathbb{F}}_p, \qquad 1\le  i<j  \le s .
\end{equation}
Obviously the polynomials $Q_{\rho _1}(X,Y),\ldots, Q_{\rho _{s}}(X,Y)$ are pairwise relatively prime, because they are irreducible and
satisfy~\eqref{ratio}. Polynomials $Q_{\rho _i}^{\sharp}(X,Y)$ are homogeneous of degree $d^{\sharp}$ and the following holds
$$
Q^{\sharp}(X,Y)=\rho _1^{-d^{\sharp}}Q_{\rho _1}^{\sharp}(X,Y)=\ldots = \rho _{s}^{-d^{\sharp}}Q_{\rho _{s}}^{\sharp}(X,Y).
$$ 
So, we have
$$
Q_{\rho _1}(X,Y)\cdot\ldots\cdot Q_{\rho _s}(X,Y)\mid \Psi(X,Y),
$$
consequently, 
$$
Q_{\rho _1}^{\sharp}(X,Y)\cdot\ldots\cdot Q_{\rho _{s}}^{\sharp}(X,Y)\mid \Psi^{\sharp}(X,Y).
$$
Since 
$$
Q_{\rho _1}^{\sharp}(X,Y)\cdot\ldots\cdot Q_{\rho _{s}}^{\sharp}(X,Y)=(\rho _1\cdot\ldots\cdot \rho _{s})^{d^{\sharp}} Q^{\sharp}(X,Y)^{s}
$$
we obtain the desired result.  
\end{proof}

\begin{lemma}\label{lem-hH}
Let $G(X,Y), H(X,Y)\in \F_p[X,Y]$ be two homogeneous polynomials. 
Also suppose that $G(X,Y)$ consists of at least two nonzero monomials, $\deg H < p$ and the number of monomials of the polynomial $H(X,Y)$ does not exceed $s$ for 
some positive integer $s<p$. Then 
$$
G(X,Y)^s\nmid H(X,Y).
$$
\end{lemma}

\begin{proof}  Let us put $Y=1$. If $G(X,Y)^s\mid H(X,Y)$ then $G(X,1)^s\mid H(X,1)$. 
The polynomial $G(X,1)$ has at least one nonzero root. It has been proved in~\cite[Lemma~6]{HBK}    
that  such a polynomial $H(X,1)$ cannot have a nonzero root of order $s$
and the result follows. 
\end{proof}

\begin{lemma}\label{lem:AB Pk}
If $AB<t/\gmax$ and $\deg \Psi < p$ then for the polynomial $P(X,Y)$ given by~\eqref{eq:P} we have
$$
P(X,Y)\nmid \Psi(X,Y).
$$
\end{lemma}

%%%%%%%%%%
\subsection{Derivatives on some curves}

There we study derivatives on the algebraic curve and define some special differential operators. 
Through this section we use 
$$
 \frac{\partial} {\partial X}, \quad  \frac{\partial} {\partial Y}
 \mand  \frac{d }{dX}$$
 for standard partial derivatives with respect to $X$ and $Y$ and for a derivative  
with respect to $X$ along the curve~\eqref{eq:curve P}.
In particular 
\begin{eqnarray}\label{eq:ddx}
\frac{d}{dX}=\frac{\partial }{\partial X}+\frac{dY}{dX}\frac{\partial}{\partial Y},
\end{eqnarray}
where  by the implicit function theorem from the equation~\eqref{eq:curve P} 
we have
$$
\frac{dY}{dX}=-\frac{\frac{\partial P}{\partial X}(X,Y)}{\frac{\partial P}{\partial Y}(X,Y)}.
$$

We also  define inductively
$$
\frac{d^k}{dX^k}=\frac{d}{dX}\frac{d^{k-1}}{dX^{k-1}}
$$
the $k$-th  derivative on the curve~\eqref{eq:curve P}.

Consider the polynomials $q_k(X,Y)$ and
$r_{k}(X,Y)$, $k\in\mathbb{N}$, which are defined inductively
as
$$
q_1(X,Y)=-\frac{\partial}{\partial X}P(X,Y),\qquad r_1(X,Y)=\frac{\partial}{\partial Y}P(X,Y),
$$
and
\begin{equation}
\begin{split}
\label{eq: qkrk}
q_{k+1}(X,Y)&=\frac{\partial q_k}{\partial X}\left(\frac{\partial P}{\partial Y}\right)^2\\
&\qquad -\frac{\partial q_k}{\partial Y}\frac{\partial P}{\partial X}\frac{\partial P}{\partial Y}
-(2k-1)q_k(X,Y)\frac{\partial^2 P}{\partial X\partial Y}\frac{\partial P}{\partial Y} \\
&\qquad \qquad  \qquad \qquad  \qquad   +(2k-1)q_k(X,Y)\frac{\partial^2 P}{\partial Y^2}\frac{\partial P}{\partial X},\\
r_{k+1}(X,Y)&=r_k(X,Y)\left(\frac{\partial P}{\partial Y}\right)^2=\left(\frac{\partial P}{\partial Y}\right)^{2k+1}.
\end{split}
\end{equation}
We now show by induction that
\begin{equation}\label{eq: dk}
\frac{d^k}{dX^k}Y=\frac{q_k(X,Y)}{r_k(X,Y)}, \qquad  k\in\mathbb{N}.
\end{equation} 
The base of induction is
$$
\frac{d}{dX}Y=-\frac{\frac{\partial}{\partial X}P(X,Y)}{\frac{\partial}{\partial Y}P(X,Y)} =\frac{q_1(X,Y)}{r_1(X,Y)}.
$$
One can now easily verifies that assuming~\eqref{eq: dk} and~\eqref{eq:ddx} we have
\begin{align*}
\frac{d^{k+1}}{dX^{k+1}}Y = \frac{d}{dX}\frac{d^{k}}{dX^{k}}Y=
 \frac{d}{dX} \frac{q_k(X,Y)}{r_k(X,Y)}
=\frac{q_{k+1}(X,Y)}{r_{k+1}(X,Y)}, 
\end{align*}
where $q_{k+1}$ and $r_{k+1}$ are given by~\eqref{eq: qkrk}, which
concludes the induction  and proves the formula~\eqref{eq: dk}.
 
The implicit function theorem gives us the derivatives
$\frac{d^{k+1}}{dX^{k+1}}Y$ at  a point $(X,Y)$ on the algebraic
curve~\eqref{eq:curve P}, if the denominator $r_{k}(X,Y)$ is not
equal to zero. Otherwise $r_k(X,Y)=0$ if and only if the
following system holds
$$
 \frac{\partial{}}{\partial{Y}}P(X, Y)= P(X,Y)=0.
 $$
 
Let us  give the following estimates

\begin{lemma}\label{d:deriv}
For all integers $k \ge 1$, the degrees of the polynomials $q_k(X,Y)$ and $r_k(X,Y)$ satisfy  the  bounds
\begin{align*}
& \deg_X q_k \le (2k-1)m-k,\qquad \deg_Y q_k \le (2k-1)n-2k+2,\\
&\deg_X r_k \le (2k-1)m, \qquad \deg_Y r_k \le (2k-1)(n-1).
\end{align*}
\end{lemma}

\begin{proof} 
Direct calculations show that
$$\deg_X q_1\le m-1 \mand \deg_Yq_1\le n,
$$ 
and using~\eqref{eq: qkrk} (with $k-1$ instead of $k$) and examining the degree of each term,  we obtain the  inequalities
\begin{align*}
&\deg_X q_k \le\deg_Xq_{k-1} +2m-1\le  (2k-1)m-k,\\
& \deg_Y q_k \le\deg_{y}q_{k-1} +2n-2\le  (2k-1)n-2k+2.
\end{align*}
We now obtain the desire bounds on $ \deg_X q_k$ and $\deg_Y q_k$ by induction.

For the polynomials $r_{k}$ the statement is obvious.
\end{proof} 

\begin{lemma}\label{P:div}
Let $Q(X,Y)\in \mathbb{F}_p[X,Y]$ be a polynomial such that
\begin{equation}\label{CondQ}
\deg_X Q(X,Y)\le \mu,\quad \deg_Y Q(X,Y)\le \nu, 
\end{equation}
and $P(X,Y)\in \mathbb{F}_p[X,Y]$ be a polynomial such that
$$
\deg_X P(X,Y)\le m,\quad \deg_Y P(X,Y)\le n.
$$
Then the divisibility condition
\begin{equation}\label{DivPQ}
P(X,Y)\mid Q(X,Y)
\end{equation}
on the coefficients of the polynomial $Q(X,Y)$ is equivalent to  a certain system
of not more than
%%  $n((\nu-n+2)m+\mu)\le 
$(\mu+\nu+1)mn$  homogeneous linear algebraic equations in coefficients of $Q(X,Y)$ as variables.  
\end{lemma}

\begin{proof}  The dimension of the vector space $\cL$ of polynomials $Q(X,Y)$ that satisfy~\eqref{CondQ} is equal to $(\mu+1)(\nu+1)$.
Let us call the vector subspace of polynomials $Q(X,Y)$ that satisfy \eqref{CondQ} and~\eqref{DivPQ} by $\widetilde {\cL}$. Because $Q(X,Y)=P(X,Y)R(X,Y)$  where the polynomial $R(X,Y)$ is such that
\begin{equation}\label{eq: Poly R}
\deg_X R(X,Y)\le \mu-m \mand \deg_Y R(X,Y)\le \nu-n,
\end{equation}
then the vector space $\widetilde {\cL}$ is isomorphic
 to the vector space of the coefficients of the polynomials $R(x,y)$
satisfying~\eqref{eq: Poly R}. 
The dimension of the vector space $\widetilde {\cL}$ is equal to 
$$
\dim \widetilde {\cL} = (\mu-m+1)(\nu-n+1).
$$ 
It means that the subspace $\widetilde {\cL}$ of the space $\cL$ is given by a system of 
\begin{align*}
(\mu+1)(\nu+1)&-(\mu-m+1)(\nu-n+1)\\
&=\mu n +\nu m-mn+m+n+1
\le(\mu+\nu+1)mn
\end{align*}
homogeneous linear algebraic equations.
\end{proof}

As in~\cite{MaVy}, we now consider the differential operators: 
\begin{equation}\label{diff-oper}
D_k=\left(\frac{\partial P}{\partial Y}\right)^{2k-1}X^kY^k\frac{d^k}{d X^k},\qquad k \in\N, 
\end{equation}
where,  as before, $\frac{d^k}{d X^k}$ denotes the $k$-th derivative on the algebraic curve~\eqref{eq:curve P} with the local parameter $X$. We note now that the derivative of a polynomial in two variables along a curve is a rational function. As 
one can see  from the inductive formula for $\frac{d^k}{d X^k}$, the result of applying any  operator $D_k$ to a polynomial in two variables is again a polynomial in two variables.

Consider non-negative integers $a, b, c$ such that $ a<A,\, b<B,\, c<C.$  From the formulas~\eqref{eq: dk}
 for derivatives on the algebraic curve~\eqref{eq:curve P} we obtain by induction the following relations
\begin{equation}\label{diff-oper-rel}
\begin{split}
& D_k \left(\frac{X}{Y}\right)^aX^{bt}Y^{(c+1)t}=R_{k,a,b,c}(X,Y)\left(\frac{X}{Y}\right)^{a}X^{bt}Y^{(c+1)t},\\
& D_k \Psi(X,Y)|_{X,Y\in \cF_i}=R_{k,i}(X,Y)|_{X,Y\in \cF_i},
\end{split}
\end{equation}  
where $\cF_i$ from formula~\eqref{eq: Fi E},
\begin{equation}
\label{Rki}
\begin{split}
R_{k,i}(X,Y) &= \sum_{0 \le a < A} 
\sum_{0 \le b< B} \\
&\qquad \sum_{0 \le c<C} \omega_{a,b,c}R_{k,a,b,c}(X,Y)\left(\frac{X}{Y}\right)^{a}\lambda_i^{-bt}\mu_i^{-(c+1)t}
\end{split}
\end{equation} 
for some coefficients $\omega_{a,b,c}\in \F_p$, $a<A$, $b<B$, $c<C$, and $\lambda_i,\mu_i$ from~\eqref{eq: Fi E}.

We now define 
\begin{equation}
\label{eq: poly R-tilde}
\widetilde{R}_{k,i}(X,Y)=Y^{A-1}R_{k,i}(X,Y). 
 \end{equation}

\begin{lemma}\label{lem-R}
The rational functions 
$R_{k,a,b,c}(X,Y)$ and $\widetilde{R}_{k,i}(X,Y)$, given by~\eqref{diff-oper-rel}
and~\eqref{eq: poly R-tilde}, are polynomials of degrees 
$$
\deg_X R_{k,a,b,c}\le  4km,  \qquad
\deg_Y R_{k,a,b,c}\le  4kn,
$$
and
$$
\deg_X \widetilde{R}_{k,i} \le A+4km,\qquad \deg_Y \widetilde{R}_{k,i}\le A+4kn.
$$
\end{lemma}

\begin{proof}   We have
\begin{equation}\label{deriv-sumform}
\begin{split}
\frac{d^k}{dX^k}X^{a+bt}Y^{(c+1)t-a}
&=\sum_{(\ell_1,\ldots,\ell_s)} C_{\ell_1,\ldots,\ell_s} X^{a+bt-k+\sum_{i=1}^s \ell_i}\\
& \qquad \quad Y^{(c+1)t-a-s}\left(\frac{d^{\ell_1}Y}{dX^{\ell_1}}\right)\ldots\left(\frac{d^{\ell_s}Y}{dX^{\ell_s}}\right), 
\end{split}
\end{equation}
where $(\ell_1,\ldots,\ell_s)$ runs through the  all $s$-tuples of positive integers   
with $\ell_1+\ldots+\ell_s\le k$, $s=0,\ldots,k$ and $C_{\ell_1,\ldots,\ell_s}$ are some constant coefficients.

By the formula~\eqref{deriv-sumform} and the form of the operator~\eqref{diff-oper} 
we obtain that $R_{k,a,b,c}(x,y)$ are polynomials and $R_{k,i}(x,y)$ are rational functions.
Actually, from the formulas~\eqref{deriv-sumform} and~\eqref{eq: dk} we easily obtain that the denominator of 
$$
\frac{d^k}{dX^k} \left(\frac{X}{Y}\right)^aX^{bt}Y^{(c+1)t}
$$
divides $\left(\frac{\partial P}{\partial Y}(X,Y) \right)^{2k-1}$. We obtain that $R_{k,a,b.c}(X,Y)$ are polynomials. From the formula~\eqref{Rki} we obtain that $R_{k,i}$ is a rational function with denominator divided by $Y^{A-1}$. Consequently, $\widetilde{R}_{k,i}$ are polynomials.

The result now follows from Lemma~\ref{d:deriv} and the 
formulas~\eqref{diff-oper} and~\eqref{diff-oper-rel}. \end{proof}

%%%%%%%%%%
\subsection{Multiplicities points on some curves}

\begin{lemma}\label{deriv:k-th}
If $P(X,Y)\mid\Psi(X,Y)$ and $P(X,Y)\mid D_j\Psi(X,Y)$, $j=1,\ldots,k-1$, then at least one of the following alternatives holds: 
\begin{itemize}
\item either $(x,y)$ is a root of  order at least $k$ of  $\Psi(X,Y)$ on
the algebraic curve~\eqref{eq:curve P};
\item or $(x,y) \in \cM_{sing}$.
\end{itemize}
\end{lemma}

\begin{proof} If $D_j\Psi(X,Y)$ vanishes on the curve $P(X,Y)= 0$, then either
\begin{equation}
\label{eq:altern1}
\frac{d^j}{dX^j}\Psi (x,y)=0,
\end{equation}
where, as before, $\frac{d^j}{d X^j}$ is $j$-th derivative on the algebraic curve~\eqref{eq:curve P} with the local parameter $X$, 
or
\begin{equation}
\label{eq:altern2}
xy=0, 
\end{equation}
or 
\begin{equation}
\label{eq:altern3}
\frac{\partial P}{\partial Y}(x,y)=0,
\end{equation}
on the curve~\eqref{eq:curve P}. 

If we have~\eqref{eq:altern1} 
for $j=1,\ldots,k-1$  and also $\Psi(x,y)=0$ 
 then the pair $(x,y)$ satisfies  the first case of conditions of Lemma~\ref{deriv:k-th}.

If we have~\eqref{eq:altern2} or~\eqref{eq:altern3} 
on the curve~\eqref{eq:curve P}  then the pair $(x,y)$ satisfies  the second case of conditions of Lemma~\ref{deriv:k-th}. 
\end{proof}

\section{Small divisors of integers}
%SK This part is not related to Section 2

\subsection{Smooth numbers} As usual, we say that a positive integer is   $y$-smooth
if it is composed of prime numbers  up to $y$. Then we denote by $\psi(x,y)$
the number of $y$-smooth positive integers that are up to $x$.
Among a larger variety of bounds and asymptotic formulas for  $\psi(x,y)$,
see~\cite{Gran, HT, Ten},   the most convenient for our applications bound is given 
by~\cite[Theorem~5.1]{Ten}.

\begin{lem}\label{lem:psi} There is an absolute constant $c_0$ such that  
for any fixed real positive $x\ge y \ge 2$ we 
have 
$$
\psi(x,y) \le c_0 e^{-u/2} x
$$
where 
$$
u = \frac{\log x}{\log y}.
$$
\end{lem}

\subsection{Number of small divisors of integers}

For a real $z$ and an integer $n$ we use $\tau_z(n)$ to denote the number of
integer positive divisors $d \mid n$ with $d \le z$. We present a bound on  $\tau_z(n)$
for small values of $z$ (which we put in a slightly more general form than we need
for our applications).

\begin{lem}\label{lem:tauz} There is an absolute constant $C_0$ such that 
for any fixed real  positive $\ve < 1$ % and $b > 2 \log (1/\ve)$
%SK there is $b(\ve)$  and $n(\ve)$  such that if $n\ge n(\ve)$ and
there is $n(\ve)$  such that if $n\ge n(\ve,b)$ and
    $z \ge  (\log n)^{2 \log (1/\ve)}$ then
    $$
    \tau_z(n) \le C_0 \ve z.
    $$
\end{lem}

\begin{proof} 
 Let $s$ be the number of all distinct prime divisors of $n$ and let $p_1, \ldots, p_s$
 be the first $s$ primes. We note that
\begin{equation}
\label{eq:tau psi}
\tau_z(n) \le \psi(z, p_s).
\end{equation}

By the prime number theorem we have $n \ge p_1\ldots p_s =  \exp(p_s + o(p_s))$
and thus
\begin{equation}
\label{eq:ps z}
p_s  \ll  \log n \le z^{1/b}.
\end{equation}
where $b = 2 \log (1/\ve)$. 
Combining Lemma~\ref{lem:psi}  with~\eqref{eq:tau psi} and~\eqref{eq:ps z}
we see that 
$$
\tau_z(n)  \le \psi(z,z^{1/b+o(1)})\le c_0 e^{-b/2+o(1) } z    = (c_0 + o(1)) e^{-b/2}z \le  C_0 \ve z
$$
for any $C_0> c_0$  (where $c_0$ is as in Lemma~\ref{lem:psi}), 
provided that $n$ and thus $z$ are large enough. 
\end{proof}

\section{Proof of  Theorem~\ref{thm:MV}}

\subsection{Preliminary estimates} 
We define the following parameters:
$$
A=\fl{\frac{t^{2/3}}{\gmax h^{1/3}}} ,\quad B=C=\fl{ h^{1/3}t^{1/3}}, 
\qquad 
D=\fl{ \frac{t^{2/3}}{4 \gmax h^{1/3} mn}}
$$
The exact values of $A$, $B$, $C$ and $D$ play no role until the optimization 
step at the very end of the proof. However it is important to note that their choice ensures~\eqref{eq:aAB t}
and~\eqref{eq: deg Psi} below.  

If $P_i(x,y)=0$ for at least one $i=1,\ldots,h$, then
\begin{equation}\label{Deriv-3}
D_k\Psi(x,y)=0,\qquad  (x,y)\in \bigcup_{i=1}^h\cF_i, 
\end{equation}
with  the operators~\eqref{diff-oper}, 
where the sets $\cF_i$ are as in~\eqref{eq: Fi E}. The condition (\ref{Deriv-3}) 
is given by a system of linear homogeneous algebraic equations in the variables $\omega_{a,b,c}$. 
The number of equations can be calculated by means of 
Lemmas~\ref{P:div} and~\ref{lem-R}.
To satisfy the condition~\eqref{Deriv-3} for some $k$ we have to make sure that the polynomials $\widetilde{R}_{k,i}(X,Y)$, $i=1,\ldots,h$, given by~\eqref{eq: poly R-tilde}, vanish identically on the curve~\eqref{eq:curve P}. The bi-degree of  $\widetilde{R}_{k,i}(X,Y)$ is given by Lemma~\ref{lem-R}: 
$$
\deg_X \widetilde{R}_{k,i} \le A+4km,\qquad \deg_Y \widetilde{R}_{k,i}\le A+4kn.
$$
The number of equations on the coefficients that give us the vanishing of polynomial $\widetilde{R}_{k,i}(X,Y)$ on the curve~\eqref{eq:curve P} is given by Lemma~\ref{P:div} and is equal to
$(\mu+\nu+1)mn$, 
where $\mu,\nu$ are as Lemma~\ref{P:div} and
$$
\mu \le A+4km, \quad \nu \le A+4kn.
$$
Finally, the condition~\eqref{Deriv-3} for some $k$ is given by $h(\mu+\nu+1)mn \le mnh(2A+4k(m+n))$ linear algebraic homogeneous equations.
Consequently, the condition~\eqref{Deriv-3} for all $k=0,\ldots,D-1$ is given by the system of 
$$L = hmn\sum_{k=0}^{D-1}\left(4k(m+n)+2A+1\right)$$
linear algebraic homogeneous equations in variables $\omega_{a,b,c}$.
Now it is easy to see that
\begin{align*}
L&=h\left((2A+1)Dmn+2nm(m+n)D(D-1)\right)\\
& \le  2hADmn+2hmn(m+n)D^2 = 2hmn(AD+(m+n)D^2).
\end{align*}

\subsection{Optimization of parameters} 
The system has a nonzero solution if the number of equations is less than to the number of variables, in particular, if 
\begin{equation}\label{number-of-equ-2}
2hmn(AD+(m+n)D^2)<ABC,
\end{equation}
as we have $ABC$ variables.
It is easy to get an upper bound for the left hand side of~\eqref{number-of-equ-2}. 
For sufficiently large  $t>c_0(m,n)$, where $c_0(m,n)$ is some constant depending only on $m$ and $n$, 
we have
\begin{equation}\label{system_has_a_solution}
\begin{split}
2hmn&(AD+(m+n)D^2)\\
& <2hmn\left(\frac{h^{-1/3}t^{2/3}}{\gmax}\frac{h^{-1/3}t^{2/3}}{4mn\gmax}+(m+n)\frac{h^{-2/3}t^{4/3}}{16m^2n^2\gmax^2}\right)\\
&<\frac{3}{4}\frac{h^{1/3}t^{4/3}}{\gmax^2}. %%<ABC.  
\end{split}
\end{equation}
Assuming that $c_0(m,n)$ is large enough, we obtain 
$$
ABC=\left\lfloor \frac{h^{-1/3}t^{2/3}}{\gmax}\right\rfloor \lfloor h^{1/3}t^{1/3}\rfloor^2>\frac{3}{4}\frac{h^{1/3}t^{4/3}}{\gmax^2}, 
$$
which together with ~\eqref{system_has_a_solution} implies~\eqref{number-of-equ-2}.

It is clear that 
\begin{equation}\label{eq:aAB t}
\gmax AB\le t.
\end{equation}
We also require that the degree of the polynomial $\Psi(x,y)$ should be less than $p$,
\begin{equation}\label{eq: deg Psi}
\deg\Psi(x,y)\le (B-1)t+Ct<p.
\end{equation}
Actually, the inequality $(B-1)t+Ct<2h^{1/3}t^{4/3}<p$ is satisfied because $t<\frac{1}{2}p^{3/4}h^{-1/4}$.

Finally, recalling Lemmas~\ref{lem-nonzero}, \ref{lem-hH} and~\ref{lem:AB Pk} and also the irreducibility of the polynomial $P(x,y)$, 
we see that $P_k(X,Y)$ and $\Psi(X,Y)$ are co-prime. Hence, by
Lemmas~\ref{lem:Msing} and~\ref{deriv:k-th} and the inequality~\eqref{eq:Nh} we obtain that $N_h$ satisfies the  inequality
\begin{align*}
N_h & \le \# \cM_{sing} + (m+n)\frac{(B+C-1)t}{D}\\
&< (m+n)^{2} + (m+n) \frac{2h^{1/3}t^{4/3}}{\fl{h^{-1/3}t^{2/3}/(4mn\gmax)}}\\
 & <12mn(m+n) \gmax h^{2/3}t^{2/3}
\end{align*}
for sufficiently large  $t>c_0(m,n)$, which concludes the proof.

\section{Proof of Theorem~\ref{thm:except set2}}

\def\X{\mathbf X}

\subsection{Outline of the proof} Before giving technical details we first outline the sequence
of steps
\begin{itemize}
\item  We consider the set $\cR=\cM_p \setminus \cC_p$  and show that
if  it is large then by Lemma~\ref{lem:tauz}  there is a large set $\cL \subseteq \cR$
%elements of the same order $t_0$ (which is also large).
elements of large orders.

%\item  Each element of $\cL$ has an orbit of size $t_0$ which is also in $\cR$.
\item  Each element $x\in\cL$ has an orbit of size $\ge t(x)/2$ which is also in $\cR$.

\item Using Conjecture~\ref{conj:eq_groups}, we estimate the size of  intersections of these orbits for distinct elements $x_1, x_2 \in \cL$.

\item We conclude that all intersections together are small and so to fit them all in $\cR$
the size of $\cR$ must be even larger than we initially assumed.

\end{itemize}

\subsection{Formal argument} 
We always assume that $p$ is large enough. Define the mapping
$$
\cT_0 \(x, y, z\) \mapsto \(x, z, 3xz-y\)
$$
where $\cT_0 = \Pi_{1,3,2}\circ \cR_2$ is the composition of the permutations
$$
\Pi_{1,3,2} = (x, y,z)  \mapsto  (x, z,y)
$$
and the involution
$$
\cR_2: (x, y,z) \mapsto (x, 3xz - y,  z)
$$
 as in the above.

Therefore the orbit $\Gamma(x, y, z)$ of $(x, y,z)$ under the
above group of transformations $\Gamma$ contains, in
particular the triples $(x, u_n, u_{n+1})$, $n = 1, 2, \ldots$,
 where the sequence $u_n$ satisfies a binary linear recurrence relation
\begin{equation}
\label{eq:bin rec}
 u_{n+2} = 3xu_{n+1} - u_n, \qquad n = 1, 2, \ldots,
\end{equation}
 with the initial values, $u_1=y$, $u_2 = z$.
 This also means that  $\Gamma(x, y, z)$ contains all triples obtained
 by the permutations of the elements in  $(x, u_n, u_{n+1})$.

 Let $\xi, \xi^{-1} \in \F_{p^2}^*$ be the roots of the characteristic polynomial
 $Z^2 -3xZ +1$ of  the recurrence relation~\eqref{eq:bin rec}.
 In particular $3x=\xi+\xi^{-1}$. Then, it is easy to see that unless  $(x, y, z)=(0,0,0)$,
 which we eliminate from the consideration, the sequence $u_n$ is
 periodic with period $t(x)$ which   is the order of $\xi$ in $\F_{p^2}^*$.

Let $B$ be a fixed positive number to  be chosen later. 
We denote
$$M_0=(\log p)^B,\quad M_1 = M_0^{1/4}/3
= (\log p)^{B/4}/3.$$

%We now fix some $\ve>0$ and denote
%$$M_0=\exp((\log p)^{2/3+\ve}),\quad M_1 = M_0^{1/6}/2
%> \exp((\log p)^{2/3+\ve/2}).$$

Assume that the remaining set of nodes $\cR=\cM_p \setminus \cC_p$
is of size $\# \cR>M_0$. Note that if $(x,y,z)\in \cR$ then
also  $(y,x,z)\in \cR$ and for any $x,y$ there are at most two
values of $z$ such that $(x,y,z)\in \cR$.
Therefore,  there are more than  $(M_0/2)^{1/2}$ elements $x\in\F_p^*$ with
$(x,y,z)\in \cR$ for some $y, z \in \F_p$.

Since there are obviously at most $T(T+1)/2$ elements $\xi \in \F_{p^2}^*$
of order at most $T$
we conclude that there is a
triple  $(x^*,y^*,z^*)\in \cR$ with
\begin{equation}
\label{eq:large t}
t(x^*)>\sqrt{(M_0/2)^{1/2}} >2M_1,
\end{equation}
where $ t(x^*)$ is the period of the sequence $u_n$  which is defined as in~\eqref{eq:bin rec} with
respect to $(x^*,y^*,z^*)$. 

Then the
orbit  $\Gamma(x^*,y^*,z^*)$ of this triple has at least $2M_1$ elements. Let
$M$ be the cardinality of the set $\cX$ of projections
along the first components of all triples  $(x,y,z) \in  \Gamma(x^*,y^*,z^*)$.
Since the orbits are closed under the permutation of coordinates,
and permutations of the triples
$$
(x^*, u_n, u_{n+1}), \qquad n = 1,   \ldots, t(x^*),
$$
where as above the sequence $u_n$  is defined as in~\eqref{eq:bin rec} with
respect to $(x^*,y^*,z^*)$ and  $ t(x^*)$ is its  period, 
produce the same projection no more than twice
we obtain
\begin{equation}
\label{eq:M and t}
M \ge \frac{1}{2}  t(x^*).
\end{equation}
Recalling~\eqref{eq:large t}, we obtain
\begin{equation}
\label{eq:large M}
M > M_1 = (\log p)^{B/4}/3.
\end{equation}
%SK2 We also notice, that by the bound~\eqref{eq:except set} we also have
Using that $(x,y,z)\not\in \cM_p$, we notice, that by the bound~\eqref{eq:except set}
\begin{equation}
\label{eq:small M}
M = p^{o(1)}.
\end{equation}

For $t\mid p^2-1$
we denote $g(t)$ the number of $x\in \cX$ for which the period of the sequence 
 $u_n$  defined as in~\eqref{eq:bin rec}  satisfies $t(x)=t$. Observe that
$$
\sum_{t\mid p^2-1}  g(t) =M.
$$
The same argument as used in the bound~\eqref{eq:M and t}
implies that
\begin{equation}
\label{eq:M and t(2)}
g(t) = 0 \quad\text{for}\quad t>2M.
\end{equation}

We apply Lemma~\ref{lem:tauz} with 
\begin{equation}
\label{eq:eps}
\ve =\frac{1}{40AC_0},
\end{equation}
where $A$ is
a bound from Conjecture~\ref{conj:eq_groups} and $C_0$ is as in Lemma~\ref{lem:psi}. 
take 
\begin{equation}
\label{eq:B}
B=16 \log (1/\ve) + 1.
\end{equation}
Since $g(t)<t$ for any $t$ and also since due to~\eqref{eq:large M} we have 
$$
4\sqrt{AM}>(\log p)^{B/8}\ge (\log (p^2-1))^{2\log (1/\ve)},
$$
by Lemma~\ref{lem:tauz},
\begin{align*}
\sum_{\substack{t\le4\sqrt{AM}\\ t\mid p^2-1}} g(t) &
< \sum_{\substack{t\le4\sqrt{AM}\\ t\mid p^2-1}} t \le 
4\sqrt{AM} \tau_{4\sqrt{AM}} (p^2-1)\\
&  \le C_0 \ve  (4\sqrt{AM})^2 = 0.4M. 
\end{align*}
Hence, we conclude that
$$\sum_{\substack{t > 4\sqrt{AM}\\ t\mid p^2-1}}  g(t) \ge 0.6M.
$$

Let $\cL$ be the set of $x \in \cX$ with $t(x) > 4\sqrt{AM}$.
We have shown that
\begin{equation}
\label{est_L}
\#\cL\ge 0.6M.
\end{equation}

For each $x \in \cL$ we fix some
$y,z \in \F_p$ such $(x,y,z)\in\Gamma (x^*,y^*,z^*)$
and again consider   the sequence
$u_n$, $n=1, 2,\ldots $,  given by~\eqref{eq:bin rec}
and of period $t(x)=t_0$, so we consider the set
$$
\cZ(x) = \{u_n~:~n =1, \ldots, t_0\}. 
$$
Let $\cH_x$ be the subgroup of $\F_{p^2}^*$ of order $t(x)$,
and $\xi(x)$ satisfy the equation $3x=\xi(x)+\xi(x)^{-1}$. One
can easily check, using an explicit expression for binary recurrence sequences
via the roots of the characteristic polynomial, that
$$
\cZ(x) = \left\{\alpha(x) u+\frac{r (x)}{\alpha(x) u}~:~u\in \cH_x\right\},
$$
where
$$r(x)=\frac{(\xi(x)^2+1)^2}{9(\xi(x)^2-1)^2},$$ and
$\alpha(x)\in\F_{p^2}^*$.
If $\xi=\xi_0$ satisfies the equation
$$r= \frac{(\xi^2+1)^2}{9(\xi^2-1)^2},$$
then other solutions are $-\xi_0, 1/\xi_0, -1/\xi_0$.
Moreover, $3x=\xi+\xi^{-1}$
can take, for a fixed $r$, at most two values whose sum is $0$.
Since every value is taken at most twice among the elements
of the sequence $u_n$, $n =1, \ldots, t(x)$, we have
 \begin{equation}
 \label{eq:Z and t0}
\# \cZ(x)  \ge \frac{1}{2} t(x) > 2\sqrt{AM}.
\end{equation}

Now we construct a set $\cL^*\subseteq\cL$. If $x,x^*\in\cL$
and $x+x^*=0$, then we put one of the elements $x,x^*$ in $\cL^*$.
If $x\in\cL$ and $-x\not\in\cL$, then we set $x\in\cL^*$.
Due to~\eqref{est_L}, we get
\begin{equation}
\label{est_L'}
\#\cL^*\ge 0.3M.
\end{equation}
Moreover, for any distinct $x,x^*\in\cL^*$ we have $x+x^*\neq0$
and, hence, $r(x)\neq r(x^*)$.

We claim that under Conjecture~\ref{conj:eq_groups} for any distinct
$x, x^*\in \cL^*$ the inequality
 \begin{equation}
 \label{est_intersect}
\#\(\cZ(x)\bigcap \cZ(x^*)\)\le 2A
\end{equation}
holds.

Indeed, take distinct elements $x,x^*\in\cL^*$.
By $\cG$ we denote the subgroup of $\F_{p^2}^*$ generated by
$\cH_x$ and $\cH_{x^*}$. Notice that due to~\eqref{eq:small M} 
and~\eqref{eq:M and t(2)} we have
 \begin{equation}
\label{eq:small G}
\#\cG = p^{o(1)}.
\end{equation}

Next, $\#(Z(x)\cap Z(x^*)$
is the number of solutions to the equation
$$
\alpha(x) u+\frac{r (x)}{\alpha(x) u}
=\alpha(x^*) v+\frac{r (x^*)}{\alpha(x^*) v}, 
\qquad  (u,v) \in \cH_x \times \cH_{x^*},
$$  
as in the above or, equivalently,
$$
P_{x,x^*}(u,v) = 0,  \qquad (u,v) \in\cH_x\times \cH_{x^*},
$$
where
\begin{align*}
P_{x,x^*}(X,Y) = \alpha(x) ^2\alpha(x^*) X^2 Y  - \alpha(x) &  \alpha(x^*)^2 X Y^2\\
&  - \alpha(x)  r (x^*) X + \alpha(x^*) r(x)Y.
\end{align*}
The number of solutions to the last equation in $(u,v) \in \cH_x\times \cH_{x^*}$ does not exceed
the number of solutions in $(u,v) \in \cG^2$. Let $Z=X/Y$. Then the equation is reduced to
 \begin{equation}
\label{eq:Z,U}
\frac{\alpha(x) ^2\alpha(x^*)Z - \alpha(x)  \alpha(x^*)^2}
{\alpha(x)  r (x^*) Z - \alpha(x^*) r(x)} =U,
\end{equation}
where $U = Y^{-2}Z^{-1}$.

Now we are in position to use Conjecture~\ref{conj:eq_groups}.
The conditions~\eqref{eq:conj-cond} on the coefficients of linear functions in the numerator and in the
denominator of the fraction in~\eqref{eq:Z,U} are satisfied since
$\alpha(x)\neq0$, $\alpha(x^*)\neq0$, and $r(x)\neq r(x^*)$.  

 Also, for large $p$
we have $\#\cG\le p^{\ve_0}$ due to~\eqref{eq:small G}. By Conjecture
\ref{conj:eq_groups}, equation~\eqref{eq:Z,U} has at most $A$ solutions in $Z,Y$.
For each solution thewre are at most two possible values of $Y$. Fixing $Y$,
we determine $X$. So, the inequality~\eqref{est_intersect} holds.

Denote
$$h=[\sqrt{M/A}]+1.$$
Due to~\eqref{eq:large M} and~\eqref{est_L'} we have $\#\cL^*\ge h$
provided that $p$ is large enough. We choose
$h$ elements $x_1,\dots,x_h$ from $\cL^*$.
It follows from~\eqref{est_intersect} that
$j=1,\ldots,h$ we have
 $$
\sum_{i=1}^{j-1} \#\(\cZ\(x_{j}\) \bigcap \cZ\(x_{i}\) \) \le 2(j-1)A.
$$
which implies, by~\eqref{eq:Z and t0}
$$
\#\(\cZ\(x_{j}\)\setminus\bigcup_{i=1}^{j-1}\cZ\(x_{i}\) \) \ge
2\sqrt{AM} - 2(j-1)A.
$$

Observe that
$$\#\left(\bigcup_{j=1}^h \cZ\(x_{j}\)\right)
=\sum_{j=1}^h\#\(\cZ\(x_{j}\)\setminus\bigcup_{i=1}^{j-1}\cZ\(x_{i}\) \).$$
Hence,
\begin{align*}
\#\left(\bigcup_{j=1}^h \cZ\(x_{j}\)\right) &> 2\sqrt{AM}h - (h-1)hA\\
&=(2\sqrt{AM} - (h-1)A)h\\
& > (2\sqrt{AM} - \sqrt{AM}) \sqrt{M/A} >M,
\end{align*}
but this inequality contradicts the definition of $M$. 
Together with the choice of $B$ given by~\eqref{eq:eps} and~\eqref{eq:B}, 
this concludes the proof.  

 \section{Comments}
 \label{sec:comm}

Let $P(n)$ be the largest primitive prime divisor of $2^n-1$, 
that is,  the largest prime  which  divides $2^n-1$, but does not divide any of the 
numbers $2^d-1$ for $1\le d<n$. Note that $P(n) \equiv 1 \pmod n$. 
 By a striking result of Stewart~\cite[Theorem~1.1]{Stew} 
we have
$$
P(n) \ge n \exp\(\frac{\log n}{104 \log \log n}\)
$$
provided that $n$ is large enough. It is also natural to assume that 
$\log P(n)/\log n\to\infty$ for $n\to\infty$. However for us a weaker assumption is sufficient.
Namely as assume that 
$$
\limsup \frac{ \log P(24m)}{\log m}  = \infty. 
$$
We then take  $n=24m$, $m\in\N$, and $p=P(n)$ such that $n= p^{o(1)}$. 
Then $p\equiv 1\pmod {24}$. Since $2$ is a quadratic residue modulo $p$,
we can take $\xi\in\F_p$ such that $\xi^2=2$. We consider a group
$\cG$ generated by $\xi$. Note that $\#\cG=2n = p^{o(1)}$ as
$n\to\infty$. The group $\cG$ contains an element $\zeta_4$ of order
$4$ and an element $\zeta_6$ of order $6$. It is easy to check that
$$((\pm\zeta_4\pm1)/\xi)^8=1.
$$
Thus 
$$
(\pm\zeta_4\pm1)^{2n} = \xi^{6n} = 1.
$$
Hence $\pm\zeta_4\pm1\in\cG$. Also,
$$(\pm\zeta_6-1)^3=1.$$
Hence, similarly $\pm\zeta_6-1\in\cG$. Consider a set $\cD$ consisting of $9$
elements
$$\cD =\{(p-1/2), 1, -2, \zeta_4, -\zeta_4, \zeta_4-1, -\zeta_4-1, \zeta_6-1, -\zeta_6-1\}.$$
Clearly, $x\in\cG, x+1\in\cG$ for any $x\in \cD$. This shows that
probably $A$ in Conjecture~\ref{conj:eq_groups} should be at least $9$.

We also observe that  in Conjecture~\ref{conj:eq_groups} the value of $\varepsilon_0$ cannot be taken greater than $1/2$.

Indeed, suppose that $p$ is a prime and $p-1$ has a divisor $t = p^{\varepsilon_0+o(1)}$, 
as $p\to \infty$ with a fixed 
$\varepsilon_0>1/2$ (the infinitude of such primes follows instantly from~\cite[Theorem~7]{Ford}). 

Let us fix any $\alpha_{1,1},\alpha_{1,2},\alpha_{2,1},\alpha_{2,2} \in \Fp$. 
Clearly   the equation~\eqref{eq:conj} has $N = p + O(1)$ of solutions $(u,v) \in \(\F_p^*\)^2$. 
Let $\cG \subseteq \F_p^*$ be a subgroup  of order $t$.
Since $\mathbb{F}_p^*$ is the union of $(p-1)/t$ cosets $a\cG$ of   $\cG$, 
the direct product $\F_p^*\times \F_p^*$ is the union of $(p-1)^2/t^2$ products of cosets of   $\cG$.  
By the Dirichlet principle that there is at least one product $a\cG\times b\cG$ such that the number of solutions $(u,v) \in a\cG\times b\cG$ (with some $a,b \in \F_p^*$) is not less than 
$$
\frac{N}{(p-1)^2/t^2} \ge  (1+ o(1)) t^2/p  \ge  p^{2\varepsilon_0-1 + o(1)}
$$
and hence is not bounded as $p\to \infty$.  Changing the variables $\widetilde{u}=a^{-1}u$, $\widetilde{v}=b^{-1}v$ in~\eqref{eq:conj} we obtain another equation of the same type 
$$
\frac{\alpha_{1,1}a b^{-1}  \widetilde{u} -\alpha_{1,2} b^{-1} }{\alpha_{2,1}a \widetilde{u}-\alpha_{2,2}} = \widetilde{v}
$$
with an unbounded number of solutions $\(\widetilde{u}, \widetilde{v}\) \in \cG^2$. 

Finally, we note that using~\cite[Theorem~1.2]{CKSZ} one concludes that Conjecture~\ref{conj:eq_groups} 
holds (in much stronger and general form) for a sequence of primes of relative density 1. 
However this does not give any new results for the sets $\cM_p$ because, as we mentioned, 
Bourgain, Gamburd and Sarnak~\cite[Theorem~2]{BGS1} have already shown that 
Conjecture~\ref{conj:BGS} holds for an overwhelming majority of primes $p \le X$ as $X\to \infty$.

\section*{Acknowledgement}

%The authors would like to thank Peter Sarnak for the encouragement and 
%useful comments as well as to the  referee  for the careful reading and valuable  suggestions.  

This work was supported by the Australian Research Council Grants DP170100786 and DP180100201 (Shparlinski) and by the Russian Science Foundation  Grant 19-11-00001 (Vyugin).

%the Program of the Presidium of the Russian Academy of Sciences no.01 ``Fundamental Mathematics and its Applications'', 
%Grant PRAS-18-01 (Konyagin), by the Russian Fund of Fundamental Research, Grants RFBR 17-01-00515 (Ma\-ka\-rychev and %Vyugin), 
%and RFBR-CNRS 16-51-150005 (Vyugin), the  Australian Research Council, 
% Grant~DP170100786 (Shparlinski) and by the Simons-IUM fellowship  (Vyugin). 

\end{document}